\documentclass{amsart}

\usepackage{amsmath, amsthm,amssymb, amsfonts}
\usepackage{rotating}
\usepackage{url}
\usepackage{microtype}
\usepackage{cite}

\newtheorem{thm}{Theorem}
\newtheorem{conj}{Conjecture}
\newtheorem{prop}[thm]{Proposition}
\newtheorem{cor}[thm]{Corollary}
\newtheorem{lemma}[thm]{Lemma}
\newtheorem*{defn}{Definition}

\theoremstyle{remark}

\newtheorem{rem}{Remark}

\def\Z{{\mathbb Z}}
\def\Q{{\mathbb Q}}
\def\E{{\mathbb E}}
\def\Zo{{\mathbb Z}_{0}}
\def\P{{\mathbb P}}
\def\ER{{Erd\H{o}s--R\'{e}nyi }}
\def\p{{\mathfrak p}}

\def\Cl{{\operatorname{Cl}}}
\def\Hom{{\operatorname{Hom}}}

\def\Aut{{\operatorname{Aut \;}}}

\def\Sur{{\operatorname{Sur}}}

\def\coker{{\operatorname{coker}}}
\def\deg{{\operatorname{deg}}}

\def\and{{\operatorname{and}}}
\def\outdeg{{\operatorname{outdeg}}}
\def\ord{{\operatorname{ord}}}
\def\Img{{\operatorname{Im}}}

\title[Coeulerian property for random directed graphs]{Sandpile groups and the coeulerian property for random directed graphs}

\author{Shaked Koplewitz}
\address{Mathematics Department, Yale University, New Haven, CT 06511}
\email{shaked.koplewitz@gmail.com}

\begin{document}

\begin{abstract}

We consider random directed graphs, and calculate the distribution of the cokernels of their laplacian, following the methods used by Wood in \cite{mw}. As a corollary, we show that the probability that a random digraph is coeulerian is asymptotically upper bounded by a constant around $0.43$.

\end{abstract}

\maketitle

\section{Introduction}
\label{sec:intro}

In this paper, we consider a random directed graph $\Gamma$, with multiple edges permitted.

Let $L=L(\Gamma)$ be the laplacian of $\Gamma$. We define the \textbf{total sandpile group $S(\Gamma)$} to be the cokernel $\Zo^{n}/L\Z^{n}$, where $\Zo^{n}$ is the subspace of $\Z^{n}$ composed of the vectors whose elements sum to zero.

We will calculate the distribution of the total sandpile group of a random digraph $\Gamma$ by studying $L(\Gamma)$ as a random matrix. Our argument will make use of the following independence condition:

\begin{defn}
A random variable $y$ in a ring $T$ is \textbf{$\epsilon$-balanced} if for every maximal ideal $\p$ of $T$ and $r\in T/\p$ we have $\P(y\equiv r\pmod{p})\leq 1-\epsilon$.

We say that a random digraph is \textbf{$\epsilon$-balanced} if for each pair of distinct vertices $i,j\in \Gamma$, $\deg(i,j)$ is an independent $\epsilon$-balanced random variable over $\Z$.

We say that a random matrix $M$ is \textbf{$\epsilon$-balanced} if all of its entries are independent $\epsilon$-balanced random variables.
\end{defn}

\noindent In particular, an \ER random digraph with constant edge probability is $\epsilon$-balanced.

For any prime $p$, define $Q_{p}=\prod_{k=2}^{\infty}(1-p^{-k})$, and let $Y_{p}$ be the distribution on isomorphism classes of finite abelian $p$-groups such that $\P(Y_{p}\cong G)=\frac{Q_{p}}{|G||\Aut(G)|}$ for any finite $p$-group $G$. It is shown in \cite{cl} that this is a probability distribution.

Let $Y$ be the distribution on isomorphism classes of groups given by taking the product of the $Y_{p}$ independently; it can easily be seen that $Y$ is finite with probability $1$, and $\P(Y\cong G)=\frac{Q}{|G||\Aut(G)|}$ for any group $G$, where $Q=\prod_{p\text{ prime}}Q_{p}$. (In fact, $Q=\left( \prod_{k=2}^{\infty}\zeta(k)\right)^{-1}\approx 0.4357571$).

We will show that the distribution of the total sandpile group of a random digraph converges to $Y$, in the following sense:

\begin{thm}
\label{thm:main}
Let $\epsilon>0$, and let $\Gamma=\Gamma(n)$ be an $\epsilon$-balanced random digraph on $n$ vertices. Then for any integer $a>0$ and finite abelian group $G$ with exponent dividing $a$,
\[\lim_{n\rightarrow\infty}\P(S(\Gamma(n))\otimes(\Z/a\Z)\cong G)= \P(Y\otimes(\Z/a\Z)\cong G). \]
\end{thm}

It is worth noting that $\Gamma(n)$ need not have the same distributions for various values of $n$, so long as they are all $\epsilon$-balanced for a constant $\epsilon$ independent of $n$.

The proof of Theorem~\ref{thm:main} will follow the approach of Wood in \cite{mw}, using results proved in \cite{mw2}.

We can state this theorem more explicitly. For any group $G$ and prime $p$, we let $G_{p}$ denote the $p$-Sylow subgroup of $G$. Let $P$ be a set of primes. For any group $G$, we define $G_{P}=\prod_{p\in P}G_{p}$. 

\begin{cor}
\label{cor:main2}
Let $\Gamma(n)$ be a sequence of random digraphs as above. Let $P$ be a finite set of primes, and let $G$ be a group such that all the prime divisors of $|G|$ are in $P$. Then
\[\lim_{n\rightarrow\infty}\P(S(\Gamma(n))_{P}\cong G) = \frac{\prod_{p\in P}Q_{p}}{|G||\Aut(G)|}. \]
\end{cor}

\noindent In other words, The distribution of $S(\Gamma(n))_{P}$ converges to that of $\prod_{p\in P}Y_{p}$ for any finite set of primes $P$. 

To see that this follows from Theorem~\ref{thm:main}, write $|G|=\prod_{p\in P}p^{k_{p}}$, and note that for any group $H$, $H_{P}\cong G$ if and only if $G\otimes (\Z/a\Z)\cong G$ where $a=\prod_{p\in P}p^{k_{p}+1}$.

We conjecture that these restrictions are unneccessary:

\begin{conj}
\label{conj:main}
Let $\Gamma(n)$ be a sequence of random digraphs as above, $G$ a finite abelian group. Then
\[ \lim_{n\rightarrow\infty}\P(S(\Gamma(n))=G)=\frac{Q}{|G||\Aut(G)|}. \]
\end{conj}

A digraph $\Gamma$ is \textbf{coeulerian} if the total sandpile group $S(\Gamma)$ is trivial. As a result of Corollary~\ref{cor:main2}, we get:

\begin{cor}
\label{cor:coeulerian}
Let $\Gamma(n)$ be a sequence of random digraphs as above, and let $q_{n}$ be the probability that $\Gamma(n)$ is coeulerian. Then $\limsup_{n\rightarrow\infty}(q_{n})\leq Q$.
\end{cor}

In \cite{tvp}, Van Pham asks if the probability that an \ER random digraph $\Gamma(n,q)$ is coeulerian converges to $1$ as $n\rightarrow\infty$ if $0<q<1$ is fixed. Corollary~\ref{cor:coeulerian} answers this in the negative. We note that if Conjecture~\ref{conj:main} holds, then $\lim_{n\rightarrow\infty}(q_{n})=Q$ follows. 

Similarly, we can bound the probability that the total sandpile group is cyclic:

\begin{cor}
\label{cor:cyclic}
Let $\Gamma(n)$ be a sequence of random digraphs as above, and let $c_{n}$ be the probability that $S(\Gamma(n))$ is cyclic. Then $\limsup_{n\rightarrow\infty}(c_{n})\leq Q\prod_{p\text{ prime}}\left(1+\frac{p}{(p-1)(p^{2}-1)}\right)$. This is equal to about $0.9603461$.
\end{cor}

\begin{proof}
$S(\Gamma(n))$ is cyclic only if $S(\Gamma(n))_{p}$ is cyclic for every prime $p$. But the distribution of $S(\Gamma(n))_{p}$ converges to that of $Y_{p}$, and in particular $\P(S(\Gamma(n))_{p} \text{ cyclic})\rightarrow\P(Y_{p}\text{ cyclic})$. A straightforward calculation (taking the sum of $\frac{1}{|\Aut(G)|}$ for all cyclic $p$-groups $G$) shows that this probability is equal to $Q_{p}\left(1+\frac{p}{(p-1)(p^{2}-1)}\right)$.

As in Corollary~\ref{cor:coeulerian}, the probability that $S(\Gamma(n))$ is cyclic is bounded by $\prod_{p}\P(S(\Gamma(n))_{p} \text{ cyclic})$, which is asymptotically bounded by $\prod_{p}\P(Y_{p}\text{ cyclic})=\prod_{p}Q_{p}\left(1+\frac{p}{(p-1)(p^{2}-1)}\right)=Q\prod_{p}\left(1+\frac{p}{(p-1)(p^{2}-1)}\right)$.
\end{proof}

As before, we note that if Conjecture~\ref{conj:main} holds, then in fact 
\[\lim_{n\rightarrow\infty}c_{n}=Q\prod_{p}\left(1+\frac{p}{(p-1)(p^{2}-1)}\right).\]

\subsection{Context}

The distribution of $Y$ first arose in relation to the distribution of class groups of real quadratic number fields. In \cite{cl}, Cohen and Lenstra conjectured that the $p$-parts of the class groups of real quadratic number fields are asymptotically distributed as in $Y_{p}$ for odd primes $p$. In particular, if $S_{X}^{+}$ is the set of real fundamental discriminants $D\leq X$, they conjecture that for any $p>2$,
 
\[ \lim_{X \rightarrow \infty} \frac{\#\{D\in S^{+}_{X}\:|\:\Cl(\Q(\sqrt{D})_{p})\cong  G \} }{|S^{+}_{X}|} =\frac{Q_{p}}{|G||\Aut(G)|}.\]

Furthermore, it is shown in \cite{mw2} that if $G_{n}$ is the cokernel of an $\epsilon$-balanced $n\times (n+1)$ random matrix, the distribution of $G_{n}$ converges to $Y$ in the sense of Theorem~\ref{thm:main}. In our case, the laplacian $L$ is a random $n\times n$ matrix, but as the image of $L$ is contained in the $n-1$-dimensional subspace $\Zo^{n}$, $L$  acts like a random $n-1\times n$ matrix. 

We also note that Maples proves some related results regarding random matrices in \cite{ms}.

In \cite{mw}, the analog of Theorem~\ref{thm:main} for undirected graphs is proven. However, the sandpile group of an undirected graph comes with an associated pairing, which affects the probability of a group appearing: In particular, Theorem 1.1. of \cite{mw} says that for a family of undirected \ER random graphs with constant edge probability,
\[ \lim_{n\rightarrow\infty} (\P(S(\Gamma)_{p})\cong G)=\frac{\#\{\text{symmetric, bilinear, perfect }\phi:G\times G\rightarrow \mathbb{C}^{*}\} }{|G||\Aut(G)|}\prod_{k\geq 0}(1-p^{-2k-1}). \]
\noindent for any prime $p$ and $p$-group $G$.

Directed graphs, however, have no naturally associated pairing, and thus the distribution of their sandpile groups is not affected by this additional structure. On a more concrete level, this difference is caused by the fact that laplacians of undirected graphs are random \textit{symmetric} matrices, while those of random directed graphs are not usually symmetric. 

For random undirected graphs, the probability that the sandpile group is trivial goes to $0$, and it is conjectured in \cite{pk} that the probability that the sandpile group is cyclic converges to $\prod_{p}\prod_{i=1}^{\infty}(1-p^{-1-2i})\approx 0.7935$. This conjecture (and the heuristic that leads to the explicit product formula from \cite{pk}) originally came from the final section of \cite{cn}. In \cite{mw}, Wood shows that this is an upper bound using the analog of Theorem~\ref{thm:main} for undirected graphs. 

For the total sandpile groups of directed graphs, on the other hand, we prove that these probabilities are asymptotically bounded by constants around $0.4358$ and $0.9603$ respectively, and conjecture that they converge to these limits.

Finally, we mention that \cite{ws} considers random matrices with entries whose distribution converges to being uniform over $\Z$. To be precise, it considers the limit as $k\rightarrow\infty$ of $M_{n,m,k}$, where $M_{n,m,k}$ is a random $n\times m$ matrix whose entries are independently chosen uniformly at random in ${-k,-k+1,\dots,k}$, and calculates the distribution of the Smith normal forms of such matrices. In the special case where $m=n+1$, they get results similar to ours: In particular, consider the sequence of random group distributions $G_{n}$ given by $\lim_{k\rightarrow\infty}\coker\left(M_{n,n+1,k}\right)$. Then as $n\rightarrow\infty$, $G_{n}$ converges to the distribution given in Conjecture~\ref{conj:main}.

\textbf{Acknowledgements}.
The author is grateful to Professor Lionel Levine for suggesting the question of the probability of a random graph being coeulerian at the BIRS-CMO Oaxaca workshop on Sandpile groups. The author is also grateful to Sam Payne and Nathan Kaplan for their many helpful suggestions along the way.

This work was partially supported by NSF CAREER DMS-1149054.

\section{Preliminaries}

\subsection{Notation}
For a pair of vertices $i,j\in\Gamma$, we let $\deg(i,j)$ denote the number of edges from $i$ to $j$. We denote by $\textbf{1}$ the vector whose entries are all $1$. We use $\P$ for probability and $\E$ for expectation. We also write $[n]={1,\dots,n}$.

We use $p$ for prime numbers, and $\prod_{p}$ denotes the product over all primes. We use $1$ for the trivial group when appropriate. For a pair of groups $G$ and $H$, we let $\Hom(G,H)$ and $\Sur(G,H)$ denote the set of homomorphisms and the set of surjections from $G$ to $H$, respectively.

As mentioned above, we use $\Zo^{n}$ to denote the subspace of $\Z^{n}$ of vectors whose entries sum to zero. For a group $G$, we use $G_{p}$ to denote its $p$-Sylow subgroup.

\subsection{Laplacians}
In this section we recall the definitions of laplacians and sandpile groups. For a more thorough introduction to the subject, including the alternative description in terms of chip-firing and some lovely pictures, see \cite{lp}.

A directed graph is \textbf{strongly connected} if there is a path from any vertex to any other vertex. It can be seen that for any $\epsilon$-balanced random graph is strongly connected with probability $1-O(e^{-cn})$: For any pair of vertices $i,j$, there exists an edge $(i,j)$ independently with probability at least $\epsilon$, so we can apply a standard \ER connectivity argument. Therefore, we can assume that our graph is strongly connected when calculating the distribution of the sandpile groups. We will do so for the remainder of this paper.

Let $\Gamma$ be a strongly connected digraph with vertices ${1,\dots,n}$. Define its \textbf{laplacian} $L=L(\Gamma)$ to be the $n\times n$ matrix given by:
\begin{align*}
L_{ij}=\left\{\begin{matrix}
-\deg(i,j) & \text{ for } i\neq j,\\ 
\outdeg(i)-\deg(i,i) & \text{ for } i=j, 
\end{matrix}\right.
\end{align*}
\noindent where $\outdeg(i)$ is the number of edges leaving the vertex $i$. Note that adding loops to $\Gamma$ has no effect on $L$, so our results will all hold both with and without loops. Note also that the columns of $L$ all sum to zero, so its image is contained in $\Zo^{n}$.

Define the \textbf{laplacian at the vertex $i$}, $L_{i}$ to be the $(n-1)\times(n-1)$ matrix given by removing the $i$th row and column from $L$. The \textbf{sandpile group at vertex $i$}, is defind as $S_{i}(\Gamma)=\Z^{n-1}/L_{i}\Z^{n-1}$. By the Matrix Tree Theorem for directed graphs, $|S_{i}(\Gamma)|=|\det(L_{i})|$ is the number of spanning trees oriented towards $i$. We define the \textbf{total sandpile group} $S(\Gamma)$ to be $\Zo^{n}/L\Z^{n}$. This is the greatest common divisor of the $S_{i}$, in the following sense:

\begin{prop}
\label{prop:totsandpile}
Let $S$ be the total sandpile group of $\Gamma$, and  let $\{S_{i}\}_{i\in [n]}$ be the sandpile groups at the vertices of $\Gamma$. Then there is a canonical surjection $S_{i}\rightarrow S$ for every $i$, and $|S|=\gcd(\{|S_{i}|\}_{i\in [n]})$. Furthermore, if $G$ is a group such that there is a surjection $S_{i}\rightarrow G$ for every $i$, then there is a surjection $S\rightarrow G$.
\end{prop}

To prove this, we will rely on the following result of Farrell and Levine:

\begin{thm}
\label{thm:totsand}
For every $i\in\Gamma$, there exists a canonical element $\gamma_{i}\in S_{i}$ such that $S_{i}/\left \langle \gamma_{i} \right \rangle\cong S$, where $\left \langle \gamma_{i} \right \rangle$ is the cyclic group generated by $\gamma_{i}$. Moreover,

\noindent $\gcd(\{\ord(\gamma_{i})\}_{i\in [n]})=1$.
\end{thm}

\begin{proof}
See Theorem 2.10 of \cite{fl} and the conversation that precedes it.
\end{proof}

Using this, we can prove Proposition~\ref{prop:totsandpile}:
\begin{proof}
Theorem~\ref{thm:totsand} gives a canonical surjection $S_{i}\rightarrow S$ for each $S_{i}$. Furthermore, $|S|=|S_{i}/\ord(\gamma_{i})|$ for every $i$. As $\gcd(\{\ord(\gamma_{i})\}_{i\in [n]})=1$, this implies $|S|=\gcd(\{|S_{i}|\}_{i\in [n]})$.

Let $G$ be a group such that there is a surjection $S_{i}\rightarrow G$ for each $i$. To find a surjection $S\rightarrow G$, it suffices to find surjections $S_{p}\rightarrow G_{p}$ for each prime $p$.

Let $p$ be a prime. As $\gcd(\{\ord(\gamma_{i})\}_{i\in [n]})=1$, there exists some $\gamma_{i}$ for which $p\nmid\ord(\gamma_{i})$. As $S_{i}/\left \langle \gamma_{i} \right \rangle\cong S$, this implies that $S_{p}=(S_{i})_{p}$. But $S_{i}$ surjects onto $G$, and hence $(S_{i})_{p}=S_{p}$ surjects onto $G_{p}$, which completes the proof.
\end{proof}

\subsection{Coeulerian Graphs}
\label{sec:coeulerians}

In this section we provide some background on coeulerian graphs. A more thorough exposition can be found in \cite{fl}.

A digraph is called \textbf{eulerian} if it has an eulerian tour (a closed path that traverses every edge exactly once). In \cite{fl}, Farrell and Levine show the following:

\begin{prop}[Proposition 2.12 of\cite{fl}]
The following are equivalent for a strongly connected digraph $\Gamma$:
\begin{enumerate}
\item $\ker(L(\Gamma))=\Z\mathbf{1}$.
\item $|S(\Gamma)|=|S_{i}(\Gamma)|$ for all $i\in [n]$.
\item $\Gamma$ is eulerian.
\item $S(\Gamma)\cong S_{i}(\Gamma)$ for some $i\in [n]$.
\item $S_{i}(\Gamma)\cong S_{j}(\Gamma)$ for all $i,j\in [n]$.
\end{enumerate}
\end{prop}

In other words, we see that the condition of being eulerian is equivalent to the kernel of each surjection $S_{i}\rightarrow S$ being trivial. Farrell and Levine define a digraph to be coeulerian when $S=1$. In \cite{fl}, they show:

\begin{thm}[Theorem 1.2 of \cite{fl}]
The following are equivalent for a strongly connected digraph $\Gamma$:
\begin{enumerate}
\item $\Img(L(\Gamma))=\Zo^{n}$.
\item $S(\Gamma)=1$
\item For all $i\in [n]$, $S_{i}(\Gamma)$ is cyclic with generator $\gamma_{i}$.
\item For some $i\in [n]$, $S_{i}(\Gamma)$ is cyclic with generator $\gamma_{i}$.
\end{enumerate}
\end{thm}

\noindent They also show equivalence of these conditions and another condition involving the chip-firing definition of the sandpile group.
A graph satisfying any of these equivalent conditions is said to be \textbf{coeulerian}.

\section{Proof of Theorem~\ref{thm:main}}
In this section we prove Theorem~\ref{thm:main}. The proof follows the proof of Theorem 6.2 of \cite{mw}, replacing the estimates provided for symmetric matrices in \cite{mw} with those provided for non-symmetric matrices in \cite{mw2}. 

For the rest of this paper, let $a$ be a fixed positive integer, $R=\Z/a\Z$. Let $V=(\Z/a\Z)^{n}=R^{n}$, let $\{v_{i}\}_{i=1}^{n}$ be the standard basis of $V$, and let $Z\subseteq V$ be the subspace of vectors in $V$ whose elements sum to zero (so $Z=\Zo^{n}\otimes R$). We wish to calculate the distribution of $\bar{S}=S\otimes R=Z/\bar{L}V$, where $\bar{L}=L\otimes R$.

Random variables can often be determined by their moments. For a random group $G$, the equivalent of the $n$\textsuperscript{th} moment of $G$ is the number of maps to a group $H$. In particular, $|\Hom(G,(\Z/p)^{n})|=X^{n}$ where $X=p^{p\text{-rank}(G)}$. Note that the moments of the distribution of $G$ are indexed by groups instead of integers.

In \cite{mw}, Wood proves that determining $\E(|\Sur(G,H)|)$ determines the distribution of $G$, in a sense we will make precise in Theorem~\ref{thm:dist}. Hence in order to calculate the distribution of $\bar{S}$, it suffices to calculate the moments $\E(|\Sur(\bar{S},G)|)$.

Note that $\E(|\Sur(\bar{S},G)|)$ counts surjections $F:\bar{S}\rightarrow G$. As $\bar{S}=\coker(\bar{L})$, $|\Sur(\bar{S},G)|$ is the number of surjections $F:Z\rightarrow G$ such that $F|_{\Img(\bar{L})}=0$. In other words, 
\begin{align}
\E(|\Sur(\bar{S},G)|)=\sum_{F\in\Sur(Z,G)}\P(F|_{\Img(\bar{L})}=0). 
\end{align}

The main theorem of this section will be concerened with calculating (1):

\begin{thm}
\label{thm:moments}
Let $G$ be a group with exponent dividing $a$, and let $S=S(\Gamma(n))$ be the total sandpile group of an $\epsilon$-balanced random graph on $n$ vertices. Then there exist constants $K,c>0$ depending only on $a,\epsilon$ and $G$ such that
\[\left|\E(\Sur(S,G))-\frac{1}{|G|}\right|\leq Ke^{-cn}.\]
\end{thm}

The rest of this section will be concerned with proving Theorem~\ref{thm:moments} by estimating the probabilities $\P(F|_{\Img(\bar{L})}=0)$. We will show that the majority of functions $F:V\rightarrow G$ are codes, as defined in \cite{mw} (for details, see Section~\ref{subsec:codes}, below). For these codes, the probability $\P(F|_{\Img(\bar{L})}=0)$ is close to what we expect for the probability that $\Img(\bar{L})$ lies in a generic fixed hyperplane. We will then bound the contribution of non-codes, which will allow us to estimate the sum.

\subsection{Codes and Depth}
\label{subsec:codes}
We will now recall some definitions from \cite{mw}.

Let $\sigma\subseteq [n]$. The we let $V_{\backslash\sigma}$ be the subspace of $V$ generated by $v_{i}$ for $i\notin\sigma$. 

\begin{defn}
We say that $F\in\Sur(V,G)$ is a \textbf{code} of distance $w$, if for every $\sigma\subset [n]$ with $|\sigma|<w$, we have $FV_{\backslash\sigma}=G$. In other words, $F$ is not only surjective, but would still surjective after removing any $w$ basis vectors from $V$.
\end{defn}

\noindent If $a$ is prime (so that $R$ is a field), this is equivalent to whether the transpose map $F^{*}: G^{*} \rightarrow V^{*}$ (given by $(F^{*}(\varphi))(v)=\varphi(F(v))$) is injective and has image $\Img(F^{*})\subseteq V^{*}$ a linear code of distance $w$ in the usual sense.

We would like to simply split $F$ into codes and non-codes. However, this is not delicate enough for our purposes. We will also need the notion of depth, which can be understood as a way to measure the size of the largest subgroup of $G$ for which $F$ is a code.

Let $\delta>0$ small, to be fixed later. For any integer $D$ with prime decomposition $D=\prod_{i} p_{i}^{e_{i}}$, define $\ell(D)=\sum_{i}e_{i}$.

\begin{defn}  
The \textbf{depth} of an $F\in\Hom(V,G)$ is the maximal positive $D$ for which there is a $\sigma\subset [n]$ with $|\sigma|<\ell(D)\delta n$ such that $D=[G:F(V_{\backslash\sigma})]$, or $1$ if there is no such $D$.
\end{defn}

\begin{rem}
\label{rem:depth}
In particular, if the depth of $F$ is $1$, then for any $\sigma\subset[n]$ with $|\sigma|\leq \delta n$, we have that $F(V_{\backslash\sigma})=G$ (as otherwise $\ell([G:F(V_{\backslash\sigma})])\geq 1$), and so we see that $F$ is a code of distance $\delta n$.
\end{rem}

The proof of Theorem~\ref{thm:moments} will rely first on the fact that codes are close to being uniformly distributed, in the sense that when $M$ is an $\epsilon$-balanced random matrix and $F$ is a code, $FM$ is close to being uniformly distributed in $G$. In particular for each basis vector $v_{i}$, $FM(v_{i})=0$ with probability roughly $\frac{1}{|G|}$. As there are $n$ basis vectors $v_{i}$, we get that $\P(FM=0)\approx\frac{1}{|G|^{n}}$.

We also use the fact that there are not too many non-codes, and bound the contribution of non-codes by showing that as they have greater depth they become less uniformly distributed but also less frequent, which balances out. To do this, we will split the sum in (1) and use the following results:

\begin{thm}[Theorem 5.2 of \cite{mw}]
\label{thm:countDepth}
(Counting $F$ of a given depth): There is a constant $K$ depending only on $G$ such that if $D>1$, then the number of $F\in\Hom(V,G)$ of depth $D$ is at most
\[K{n \choose{\left \lceil \ell(D)\delta n \right \rceil-1}}|G|^{n}|D|^{-n+\ell(D)\delta n}. \]
\end{thm}

In particular, we will use the following variant:

\begin{thm}[Theorem 5.3 of \cite{mw}]
\label{thm:countDepth2}
Let $pr_{2}:G\oplus R\rightarrow R$ be projection onto the second factor. There is a constant $K$ depending on $G$ such that if $D>1$, then the number of $F\in\Hom(V,G\oplus R)$ of depth $D$ such that $pr_{2}(Fv_{i})=1$ for all $i\in [n]$ is at most
\[K{n \choose{\left \lceil \ell(D)\delta n \right \rceil-1}}|G|^{n}|D|^{-n+\ell(D)\delta n} .\]
\end{thm}

\noindent The condition $pr_{2}(Fv_{i})=1$ will allow us to encode the restriction that $\Img(\bar{L})\subseteq Z$.

The above results bound the number of non-codes of depth $D$. We now wish to bound the total contribution of $\P(F|_{\Img(\bar{L})}=0)$ when $F$ isn't a code:

\begin{thm}[Theorem 2.8 of \cite{mw2}]
\label{thm:probDepth}
Let $\epsilon>0$, $\delta>0$, and let $M\in\Hom(V,V)$ be an $\epsilon$-balanced random matrix. Then there is a constant $K$ depending only on $G,\epsilon,\delta$, and $a$ such that if $F\in\Hom(V,G)$ has depth $D>1$ and $[G:F(V)]<D$, then
\[ \P(FM=0)\leq Ke^{-\epsilon n}D^{n}|G|^{-n}. \]
\end{thm}
\noindent Note that in particular, $[G:F(V)]<D$ is always true if $F(V)=G$.

Finally, codes are close to being uniformly distributed, in the following sense:
\begin{lemma}[Lemma 2.4 of \cite{mw2}]
\label{lemma:codes}
Let $\epsilon>0,\delta>0$. Let $M\in\Hom(V,V)$ be an $\epsilon$-balanced random matrix, and let $F\in \Hom(V,G)$ be a code of distance $\delta n$. Then there are $K,c>0$ depending only on $G,\epsilon,\delta$, and $a$ such that for all $n$ we have
\[ |\P(FM=0)-|G|^{-n}|\leq \frac{Ke^{-cn}}{|G|^{n}}. \]
\end{lemma}

In particular, applying this to $G\oplus R$ gives us the following corollary:
\begin{cor}
\label{cor:codes}
Let $F\in\Sur(V,G\oplus R)$ be a code of distance $\delta n$, and $M$ as before. Then 
\[ \left|\P(FM=0)-(a|G|)^{-n}\right|\leq \frac{Ke^{-cn}}{a^{n}|G|^{n}}. \]
\end{cor}

\subsection{Proof of Theorem~\ref{thm:moments}}

We are now ready to prove Theorem~\ref{thm:moments}. We follow the proof of Theorem 6.2 of \cite{mw}, replacing the results of \cite{mw} for symmetric matrices with those of \cite{mw2} for non-symmetric matrices when neccessary.

\begin{proof}
Let $X$ be an $n\times n$ random matrix over $R$ with $X_{ij}$ distributed as $\bar{L}_{ij}$ for $i\neq j$, and $X_{ii}$ distributed uniformly in $R$, with all the entries $X_{ij}$ independent. Let $F_{0}\in\Hom(V,R)$ be the map that sends each basis element $v_{i}$ to $1$. If we condition on $F_{0}X=0$, then we find that $X$ and $\bar{L}$ have the same distribution. Also, given $X$ and conditioning on the off diagonal entries, we see that the probability that $F_{0}X=0$ is $a^{-n}$ for any choice of off diagonal entries. So any choice of off diagonal entries is equally likely in $\bar{L}$ as in $X$ conditioned on $F_{0}X=0$.

For any $F\in\Hom(V,G)$, we have
\[\P(F\bar{L}=0)=\P(FX=0\:|\:F_{0}X=0)=\P(FX=0 \:\and\: F_{0}X=0)a^{n} .\]   

\noindent Let $\tilde{F}\in\Hom(V,G\oplus R)$ be the sum of $F$ and $F_{0}$.
Recall that $Z\subset V$ denotes the vectors whose coordinates sum to zero, i.e. $Z=\ker(F_{0})$. Let $\Sur^{*}(V,G)$ denote the maps from $V$ to $G$ that are a surjection when restricted to $Z$. We wish to estimate

\begin{align*}
\E(\#\Sur(\bar{S},G))&=\E(\#\Sur(Z/\Img(\bar{L}),G)) \\
&= \sum_{F\in\Sur(Z,G)}\P(F\bar{L}=0) \\
&= \frac{1}{|G|}\sum_{F\in\Sur^{*}(V,G)}\P(F\bar{L}=0) \\
&= |G|^{-1}a^{n}\sum_{F\in\Sur^{*}(V,G)}\P(\tilde{F}X=0)
\end{align*}

\noindent Note that if $F:V\rightarrow G$ is a surjection when restricted to $Z$, then $\tilde{F}$ is a surjection from $V$ to $G\oplus R$. 

We will first break this sum apart:
\begin{align*}
&|G|^{-1}a^{n}\sum_{F\in\Sur^{*}(V,G)}\P(\tilde{F}X=0) = \\
 &\frac{a^{n}}{|G|}\sum_{\underset{\tilde{F}\text{ not code of distance }\delta n}{F\in\Sur^{*}(V,G)}}\P(\tilde{F}X=0) + 
 \frac{a^{n}}{|G|}\sum_{\underset{\tilde{F}\text{ code of distance }\delta n}{F\in\Sur^{*}(V,G)}} \left(\P(\tilde{F}X=0)-\frac{1}{(a|G|)^{n}}\right) + \\
 &\frac{a^{n}}{|G|}\sum_{\underset{\tilde{F}\text{ code of distance }\delta n}{F\in\Sur^{*}(V,G)}} \left(\frac{1}{(a|G|)^{n}}\right)
\end{align*}

In order to estimate the sum, we bound the first two parts by $Ke^{-cn}$ for constants $c,K$, then show that the third part is approximately $\frac{1}{|G|}$. Throughout the proof, we will take $c$ and $\delta$ to be sufficiently small constants. We will allow $K$ to change from line to line, so long as it is a constant depending only on $G,\epsilon,a$, and $\delta$.

We begin by bounding the first part: 

\begin{align*}
&\frac{a^{n}}{|G|}\sum_{\underset{\tilde{F}\text{ not code of distance }\delta n}{F\in\Sur^{*}(V,G)}}\P(\tilde{F}X=0)& \\
&\leq\frac{a^{n}}{|G|}\sum_{\underset{D|\#G}{D>1}}\sum_{\underset{\tilde{F}\text{ of depth }D}{F\in\Sur^{*}(V,G)}}\P(\tilde{F}X=0) && \text{(By Remark~\ref{rem:depth})}\\
&\leq\frac{a^{n}}{|G|}\sum_{\underset{D|\#G}{D>1}}\#\{\tilde{F}\in\Hom(V,G\oplus R)\text{ of depth }D\:|\:pr_{2}(v_{i})=1\text{ for all }i\}Ke^{-\epsilon n}D^{n}(a|G|)^{-n} && \text{(By Theorem~\ref{thm:probDepth})} \\
&\leq\frac{a^{n}}{|G|}\sum_{\underset{D|\#G}{D>1}} K{n \choose{\left \lceil \ell(D)\delta n \right \rceil-1}}|G|^{n}D^{-n+\ell(D)\delta n}e^{-\epsilon n}D^{n}(a|G|)^{-n} && \text{(By Theorem~\ref{thm:countDepth2})} \\
&\leq K{n \choose{\left \lceil \ell(D)\delta n \right \rceil-1}}|G|^{\ell(|G|)\delta n}e^{-\epsilon n}\leq Ke^{-cn}.
\end{align*}
The last inequality holds provided that $c<\epsilon$ and $\delta$ is sufficiently small.

We now wish to bound 

\[\left|\frac{a^{n}}{|G|}\sum_{\underset{\tilde{F}\text{ code of distance }\delta n}{F\in\Sur^{*}(V,G)}} \left(\P(\tilde{F}X=0)-\frac{1}{(a|G|)^{n}}\right)\right|. \]

Recalling that by Corollary~\ref{cor:codes}, when $\tilde{F}$ is a code of distance $\delta n$, we have 

\[\left|\P(\tilde{F}X=0)-\frac{1}{(a|G|)^{n}}\right|\leq\frac{Ke^{-cn}}{a^{n}|G|^{n}}. \]

Using this, we get

\begin{align*}
\left|\frac{a^{n}}{|G|}\sum_{\underset{\tilde{F}\text{ code of distance }\delta n}{F\in\Sur^{*}(V,G)}} \left(\P(\tilde{F}X=0)-\frac{1}{(a|G|)^{n}}\right)\right|&\leq 
\frac{a^{n}}{|G|}\sum_{\underset{\tilde{F}\text{ code of distance }\delta n}{F\in\Sur^{*}(V,G)}} \left|\P(\tilde{F}X=0)-\frac{1}{(a|G|)^{n}}\right| \\
&\leq \frac{a^{n}}{|G|}\sum_{\underset{\tilde{F}\text{ code of distance }\delta n}{F\in\Sur^{*}(V,G)}} \frac{Ke^{-cn}}{a^{n}|G|^{n}} \\
&\leq \frac{a^{n}}{|G|}\left|\Sur^{*}(V,G) \right|\frac{Ke^{-cn}}{a^{n}|G|^{n}} \\
&\leq \frac{a^{n}}{|G|}\left|\Hom(V,G) \right|\frac{Ke^{-cn}}{a^{n}|G|^{n}} \\
&\leq \frac{a^{n}}{|G|}\left|G \right|^{n}\frac{Ke^{-cn}}{a^{n}|G|^{n}} \\
&\leq Ke^{-cn}.
\end{align*}

Finally, we wish to estimate 
\[\frac{a^{n}}{|G|}\sum_{\underset{\tilde{F}\text{ code of distance }\delta n}{F\in\Sur^{*}(V,G)}} \left(\frac{1}{(a|G|)^{n}}\right)=\frac{\left|\{F\in\Sur^{*}(V,G)\:|\:\tilde{F}\text{ code of distance }\delta n\} \right|} {|G|^{n+1}}.\]

To do this, note that
\begin{align*}
&\frac{\left|\{F\in\Sur^{*}(V,G)\:|\:\tilde{F}\text{ code of distance }\delta n\} \right|}{|G|^{n+1}}= \\
& \frac{|\Hom(V,G)|}{|G|^{n+1}}-
\frac{|\Hom(V,G)\backslash\Sur^{*}(V,G)|}{|G|^{n+1}}- 
\frac{\left|\{F\in\Sur^{*}(V,G)\:|\:\tilde{F}\text{ not code of distance }\delta n\} \right|}{|G|^{n+1}}.
\end{align*}

Recall that $|\Hom(V,G)|=|G|^{n}$, so $\frac{|\Hom(V,G)|}{|G|^{n+1}}=\frac{1}{|G|}$. We will show that the other parts are bounded by $Ke^{-cn}$, which will complete the proof.

To bound the second part, we see that
\begin{align*}
\frac{|\Hom(V,G)\backslash\Sur^{*}(V,G)|}{|G|^{n+1}}&=\sum_{F\in\Hom(V,G)\backslash \Sur^{*}(V,G)}|G|^{-n-1} \\
&\leq \sum_{H\lneqq G}|G|\sum_{F\in\Hom(Z,H)}|G|^{-n-1} \\
&\leq \sum_{H\lneqq G}\frac{|H|^{n-1}}{|G|^{n}} \\
&\leq \sum_{H\lneqq G}2^{-n}\leq K2^{-n} \leq Ke^{-cn}.
\end{align*}

And for the third part, we have

\begin{align*}
\frac{\left|\{F\in\Sur^{*}(V,G)\:|\:\tilde{F}\text{ not code of distance }\delta n\} \right|}{|G|^{n+1}}&=\sum_{\underset{\tilde{F}\text{ not code of distance }\delta n}{F\in\Sur^{*}(V,G)}}|G|^{-n-1}\\
&\leq\sum_{\underset{D|\#G}{D>1}}\sum_{\underset{\tilde{F}\text{ of depth }D}{F\in\Sur^{*}(V,G)}}|G|^{-n-1} \\
&\leq \sum_{\underset{D|\#G}{D>1}} K{n \choose{\left \lceil \ell(D)\delta n \right \rceil-1}}|G|^{n}D^{-n+\ell(D)\delta n}|G|^{-n-1} \\
&\leq\sum_{\underset{D|\#G}{D>1}} K{n \choose{\left \lceil \ell(|D|)\delta n \right \rceil-1}}D^{-n+\ell(D)\delta n} \\
&\leq K{n \choose{\left \lceil \ell(|G|)\delta n \right \rceil-1}}2^{-n+\ell(|G|)\delta n} \\
&\leq Ke^{-cn}.
\end{align*}

\noindent which holds for $0<c<\log (2)$ and sufficiently small $\delta>0$. 

Putting all these bounds together, we see that 

\[\left| |G|^{-1}a^{n}\sum_{F\in\Sur^{*}(V,G)}\P(\tilde{F}X=0)-\frac{1}{|G|}\right|\leq Ke^{-cn}.\]

And as $\E(\Sur(S,G))=|G|^{-1}a^{n}\sum_{F\in\Sur^{*}(V,G)}\P(\tilde{F}X=0)$, this completes the proof.
\end{proof}

\subsection{Moments Determine the Distribution}

Finally, we show that Theorem~\ref{thm:moments} implies Theorem~\ref{thm:main}. To do this, we quote two further results of \cite{mw2}:

\begin{thm}
\label{thm:dist}
(Theorem 8.3 of \cite{mw}): Let $X_{n}$ and $X^{'}_{n}$ be sequences of random finitely generated abelian groups. Let $a$ be a positive integer and let $A$ be the set of abelian groups with exponent dividing $a$. Suppose that for every $G\in A$, there exists a value $M_{G}\leq|\wedge^{2}(G)|$ (where $\wedge^{2}(G)$ is the exterior square of $G$), such that
\[ \lim_{n\rightarrow\infty} \E(\#\Sur(X_{n},G))=M_{G}. \]

\noindent Then for every $H\in A$, $\lim_{n\rightarrow\infty}\P(X_{n}\otimes \Z/a\Z\cong H)$ exists, and for all $G\in A$ we have
\[ \sum_{H\in A} \lim_{n\rightarrow\infty}\P(X_{n}\otimes \Z/a\Z\cong H)\#\Sur(H,G)=M_{G}.\]

If for every $G\in A$, we also have $ \lim_{n\rightarrow\infty} \E(\#\Sur(X^{'}_{n},G)=M_{G})$, then for every $H\in A$
\[\lim_{n\rightarrow\infty}\P(X_{n}\otimes \Z/a\Z\cong H)=\lim_{n\rightarrow\infty}\P(X^{'}_{n}\otimes \Z/a\Z\cong H).\]
\end{thm}

Let $X_n=S(\Gamma(n))$. By Theorem~\ref{thm:moments}, $ \lim_{n\rightarrow\infty} \E(\#\Sur(X_{n},G))=\frac{1}{|G|}$.  Let $X^{'}_{n}=Y$ be our random group from Section~\ref{sec:intro}. As $\frac{1}{|G|}\leq 1 \leq|\wedge^{2}(G)|$, it suffices to show that $ \lim_{n\rightarrow\infty} \E(\#\Sur(X^{'}_{n},G))=\E(\#\Sur(Y,G))=\frac{1}{|G|}$ in order to prove Theorem~\ref{thm:main}. This is shown by Wood in \cite{mw2}:

\begin{lemma}
\label{lemma:udist}
(Lemma 3.2 of \cite{mw2}): Let $Y$ be a random group distributed as above. Then for any abelian group $G$, $\E(\#\Sur(Y,G))=\frac{1}{|G|}$.
\end{lemma}

This completes the proof of Theorem~\ref{thm:main}.

\bibliography{bib}
\bibliographystyle{plain}

\end{document}